\documentclass[a4paper, 11pt]{amsart}

\usepackage{amsmath, amssymb, amsthm, mathrsfs, hyperref, color, enumerate, tikz, calligra, tikz-cd, mathtools, array}
\usepackage[capitalize]{cleveref}
\usepackage[all]{xy}
\usetikzlibrary{calc, matrix}

\newcommand{\R}{\mathbb{R}}
\newcommand{\C}{\mathbb{C}}
\newcommand{\Z}{\mathbb{Z}}
\renewcommand{\P}{\mathbb{P}}
\newcommand{\Hom}{\operatorname{Hom}}
\newcommand{\sheafHom}{\operatorname{\mathscr{H}\text{\kern -3pt {\calligra\large om}}\,}}
\newcommand{\ev}{\operatorname{ev}}
\newcommand{\ssvdots}{\raisebox{-1ex}{\smash{\vdots}}} % for spectral sequence vdots
\newcommand{\CO}{\operatorname{\mathcal{CO}^0}}
\newcommand{\at}[2]{\left.#1\right|_{#2}}
\renewcommand{\phi}{\varphi}

\theoremstyle{plain}
\newtheorem{thm}{Theorem}
\newtheorem{lem}[thm]{Lemma}
\newtheorem{prop}[thm]{Proposition}
\newtheorem{cor}[thm]{Corollary}

\theoremstyle{remark}
\newtheorem{rmk}[thm]{Remark}

\title{Monotone Lagrangians in $\C\P^n$ of minimal Maslov number $n+1$}
\author{Momchil Konstantinov}
\address{Department of Mathematics, University College London, Gower Street, London, WC1E  6BT}
\email{momchil.konstantinov.14@ucl.ac.uk}
\author{Jack Smith}
\address{St John's College, Cambridge, CB2 1TP}
\email{j.smith@dpmms.cam.ac.uk}

\begin{document}
\begin{abstract}
We show that a monotone Lagrangian $L$ in $\C\P^n$ of minimal Maslov number $n + 1$ is homeomorphic to a double quotient of a sphere, and thus homotopy equivalent to $\R\P^n$. To prove this we use Zapolsky's canonical pearl complex for $L$ with coefficients in $\Z$, and various twisted versions thereof, where the twisting is determined by connected covers of $L$. The main tool is the action of the quantum cohomology of $\C\P^n$ on the resulting Floer homologies. 
\end{abstract}
\maketitle

\section{Introduction}

\subsection{Background and statement of results}

Fix a positive integer $n$ and suppose $L \subset \C\P^n$ is a closed, connected, monotone Lagrangian submanifold of minimal Maslov number $N_L = n + 1$ (see \cref{section soft concepts} below for a review of the Maslov index and monotonicity).  It is well-known, following Seidel \cite{SeidelGraded}, that this is the maximal possible value of $N_L$ for a monotone Lagrangian in projective space.  It is attained by the standard $\R\P^n \subset \C\P^n$, but up to Hamiltonian isotopy there are no other known examples. In this note we show the following:

\begin{thm}
\label{Theorem1}
Let $L \subset \C\P^n$ be a closed, connected, monotone Lagrangian submanifold satisfying $N_L = n + 1$. Then
$L$ has fundamental group $\Z/2$ and its universal cover is homeomorphic to $S^n$.
\end{thm}
\noindent
Combined with \cite[Lemma 3]{HirschMilnorInvolutionsOfSpheres}, \cref{Theorem1} immediately implies:
\begin{cor}\label{corollary homotopy equivalence}
$L$ is homotopy equivalent to $\R\P^n$.
\end{cor}

This constitutes a step towards answering a question posed by Biran and Cornea in \cite[Section 6.2.5]{BiranCorneaQS}, informally asking whether a Lagrangian $L$ in $\C\P^n$ which ``looks like'' $\R\P^n$ must be (diffeomorphic to, Hamiltonian isotopic to) $\R\P^n$; some history of this problem is discussed in \cref{subsection previous results}. For us ``looks like'' will always mean that it is monotone of minimal Maslov number $n + 1$.  When $n = 1$ or $2$, the answer is as strong as possible: any such Lagrangian is Hamiltonian isotopic to $\R\P^n$.  This is trivial for $n=1$, whilst the $n = 2$ case follows from recent work of Borman--Li--Wu \cite[Theorem 1.3]{BormanLiWuSphericalLags}. Our \cref{Theorem1} allows us to prove diffeomorphism for $n = 3$:
\begin{cor}
Let $L\subset \C\P^3$ be a monotone Lagrangian of minimal Maslov number $4$. Then $L$ is diffeomorphic to $\R\P^3$. 
\end{cor}
\begin{proof}
By \cite[Theorem 3]{LivesayInvolutionsOfS3} any fixed-point-free involution on $S^3$ is conjugate to the antipodal map by a homeomorphism, and in dimension $3$ the topological and smooth categories are equivalent.
\end{proof}

However, we cannot easily upgrade homotopy equivalence to diffeomorphism for $n \geq 4$ as the papers \cite{CappellShaneson} ($n=4$) and \cite{HirschMilnorInvolutionsOfSpheres} ($n \geq 5$) show.

We end this discussion by noting that \cref{corollary homotopy equivalence} implies a version of the nearby Lagrangian conjecture for $\R\P^n$:

\begin{cor}
\label{corollary nearby Lagrangian}
Any closed, connected, exact Lagrangian in $T^*\R\P^n$ with vanishing Maslov class is homotopy equivalent to the zero section.
\end{cor}
\begin{proof}
Let $L$ be such a Lagrangian.  Note that $\C\P^n$ decomposes as the union of a quadric $Q$ and a disjoint Weinstein neighbourhood $U$ of the standard $\R\P^n$, so by rescaling $L$ towards the zero section if necessary we may assume it embeds in $U$ and hence in $\C\P^n$.  By considering the long exact sequence in cohomology for the triple $(\C\P^n, U, L)$, with real coefficients, the exactness and vanishing of the Maslov class of $L$ in $U$ imply that $L$ is monotone in $\C\P^n$.  (See \cref{section soft concepts} for a summary of the Maslov class and monotonicity.)

More specifically, $H^2(\C\P^n, U; \Z)$ is freely generated by the class $\alpha$ given by ``intersection with $Q$'', and there exist positive constants $A$ and $M$ such that for every disc class $\beta$ in $\pi_2(\C\P^n, L)$ the area and Maslov index of $\beta$ are given by $A\alpha(\beta)$ and $M\alpha(\beta)$ respectively.  Since the Maslov index of a line is $2(n+1)$ (twice its Chern number), we see that $M$ is in fact $n+1$.  Hence $L$ has minimal Maslov number $n+1$ and \cref{corollary homotopy equivalence} gives the result.
\end{proof}

This result was already known from the work of Abouzaid \cite{AbouzaidNearbyLagrangians}, building on Fukaya--Seidel--Smith \cite{FukayaSeidelSmith} and Nadler \cite{NadlerMicrolocalBranes}, but our approach is much more elementary.  Kragh later removed the Maslov-zero hypothesis with Abouzaid \cite{KraghRingSpectra}, and subsequently gave a simpler proof \cite{KraghHomotopyEquivalenceSerre} of a weaker statement which also implies \cref{corollary nearby Lagrangian}, but using completely different methods.

\subsection{Relation to previous works}
\label{subsection previous results}

Monotone Lagrangians $L$ in $\C\P^n$, and especially those which resemble $\R\P^n$, have been intensively studied.  Back in \cite[Theorem 3.1]{SeidelGraded} Seidel showed that for any monotone $L \subset \C\P^n$ the group $H^1(L; \Z/(2n+2))$ is non-zero (this is roughly equivalent to the fact that the minimal Maslov number of $L$ is at most $n+1$, since the mod-$(2n+2)$ reduction of the Maslov class in $H^2(X, L; \Z/(2n+2))$ lifts to $H^1(L; \Z/(2n+2))$ and it is this lift which is shown to be non-zero), and that if it's $2$-torsion then there is an isomorphism of graded $\Z/2$-vector spaces $H^*(L; \Z/2) \cong H^*(\R\P^n; \Z/2)$.  He did this by showing that the Floer cohomology of $L$ is $2$-periodic in its grading, and comparing it with the classical cohomology of $L$ via the Oh spectral sequence \cite{OhSpectralSequence} (also constructed by Biran--Cornea \cite{BiranCorneaQS}, and recapped in \cref{CanonicalComplex} below).  In particular, if $L \subset \C\P^n$ is a Lagrangian satisfying $2H_1(L; \Z)=0$ (which automatically implies that it's monotone and that its minimal Maslov number is $n+1$) then $L$ is additively a $\Z/2$-homology $\R\P^n$.

Later, Biran--Cieliebak \cite[Theorem B]{BiranCieliebak} reproved the first part of Seidel's result by introducing the important \emph{Biran circle bundle} construction, which associates to a monotone Lagrangian in $\C\P^n$ a displaceable one in $\C^{n + 1}$ and then uses the vanishing of the Floer cohomology of the latter to constrain the topology of the former via the Gysin sequence.  Combining this construction with the Oh spectral sequence, Biran \cite[Theorem A]{BiranNonIntersections} then reproved the second part of Seidel's result---the $\Z/2$-homology isomorphism---but under the hypothesis that $L \subset \C\P^n$ is monotone and of minimal Maslov number $n+1$ (he states the assumption that $H_1(L; \Z)$ is $2$-torsion but only uses the monotonicity and minimal Maslov consequences). Note that, in conjunction with the classification of surfaces, this result already shows that for $n = 2$ the Lagrangian must be diffeomorphic to $\R\P^2$.

The next major development was the introduction of the pearl complex model for Floer cohomology by Biran--Cornea \cite{BiranCorneaQS}, using which they gave another proof of the additive isomorphism $H^*(L;\Z/2) \cong H^*(\R\P^n; \Z/2)$ and showed that it is in fact an \emph{algebra} isomorphism if $H_1(L;\Z)$ is $2$-torsion \cite[Section 6.1]{BiranCorneaRigidityUniruling} (this was partially proved in \cite{BiranNonIntersections}; the $2$-torsion assumption is only used for odd $n$).  The key ingredient is the quantum module action of the hyperplane class $h$ in $QH^*(\C\P^n;\Z/2)$ on $HF^*(L,L;\Z/2)$: since $h$ is invertible, this gives an isomorphism
\[
h \mathbin{*} - : HF^*(L, L; \Z/2) \xrightarrow{\ \sim \ } HF^{*+2}(L, L; \Z/2)
\]
which subsumes both Seidel's periodicity observation and the circle bundle Gysin sequence map given by cupping with the Euler class.

The strongest results to date were then obtained by Damian \cite[Theorem 1.8 c)]{Damian}, who applied his lifted Floer theory to the circle bundle construction to show that when $n$ is odd and $2H_1(L;\Z) = 0$, $L$ must be homeomorphic to a double quotient of $S^n$.  In fact, under our weaker hypothesis---monotonicity and minimal Maslov number $n+1$---Damian's methods can be pushed to give:
\begin{align}
&\text{For odd $n$, the universal cover $\widetilde{L}$ is homeomorphic to $S^n$ and $\pi_1(L)$ is finite.}\label{eqDamianOdd}
\\ &\text{For even $n$, $\widetilde{L}$ is a $\Z/2$-homology sphere and $\pi_1(L) \cong \Z/2$.}\label{eqDamianEven}
\end{align}
We sketch these arguments in \cref{subsection Damian}, and thank the referee for pointing them out to us.

Our \cref{Theorem1} strengthens these results by showing that, regardless of the parity of $n$, $\widetilde{L}$ is homeomorphic to $S^n$ and $\pi_1(L) \cong \Z/2$.  One notable feature of Damian's approach is its reliance on the ingenious auxiliary construction of the circle bundle, which replaces our Lagrangian $L$ in $\C\P^n$ with the related Lagrangian $\Gamma_L$ in $\C^{n+1}$ that is necessarily displaceable.  Part of our motivation was to see whether one could prove the same results by directly studying the Floer theory of $L$, and this paper answers that question in the affirmative.

%----------------------------------------------------

\subsection{Idea of proof}

The proof of \cref{Theorem1} is a combination of the quantum module action and lifted Floer theory, which we discuss within the more general framework of higher rank local systems.  For many of the arguments we need to use Floer theory with $\Z$ coefficients, rather than $\Z/2$, which is problematic when $n$ is even because then $L$ is necessarily non-orientable, and hence so are the moduli spaces of holomorphic discs that we wish to count. However, we are able to work around this using the recently-introduced canonical orientations package of Zapolsky \cite{Zapolsky}, of which the present paper represents one of the first concrete applications.

Ignoring many technicalities, the idea (for $n \geq 3$) is roughly as follows.  Since $N_L \geq 4 > 2$, the Floer cohomology of $L$ can be defined with arbitrary higher rank local systems; in particular, lifted Floer theory is defined for all covers $L'$ of $L$.  For such covers, the Oh spectral sequence computing $HF^*(L, L'; R)$ contains the compactly-supported singular cohomology $H^*_c(L'; R)$ in the zeroth column of its first page, as shown in \cref{figLprime}.
\begin{figure}[ht]
\begin{tikzpicture}
\begin{scope}[xshift=-3cm]
\matrix(m)[matrix of math nodes, nodes={inner sep=0, outer sep=0, minimum height=3.2ex, text depth=0ex, minimum width=0ex}, column sep=4.5ex]
{
 &  & \ssvdots \\
 & H^n_c(L'; R) & \raisebox{1.5pt}{$\bullet$} \\
 & \ssvdots & \\
 & H^2_c(L'; R) &  \\
 & H^1_c(L'; R) &  \\
\raisebox{1.5pt}{$\bullet$} & H^0_c(L'; R) &  \\
\ssvdots &  &  \\
};
\begin{scope}[every path/.style={shorten <= 0.15cm, shorten >= 0.15cm}]
\draw[<-] (m-6-2.west)-- +(-0.75cm, 0);
\draw[->] (m-2-2.east)-- +(0.75cm, 0);
\end{scope}
\end{scope}
\end{tikzpicture}
\caption{The zeroth column of the first page of the Oh spectral sequence computing $HF^*(L, L'; R)$.\label{figLprime}}
\end{figure}
As $N_L= n+1$ all of the cohomology groups $H^*_c(L'; R)$ with $0 < * < n$ survive to the limit.  We are being deliberately vague about the choice of ring $R$ and what the terms labelled $\bullet$ are in the spectral sequence.

There is an algebra homomorphism $\CO : QH^*(\C\P^n; R) \rightarrow HF^*(L, L;R)$ called the length-zero closed--open string map \cite[Definition 2.3]{SheridanFano}, which is described in Zapolsky's framework \cite[Section 3.9.3]{Zapolsky} as the quantum module action of $QH^*$ on the unit $1_L$ in $HF^0$.  When $R=\Z/2$ we show that it is an isomorphism in degree $2$.  Since $N_L \geq 4$, this map in degree $2$ coincides with the classical restriction $i^* : H^2(\C\P^n; R) \rightarrow H^2(L; R)$, so we deduce that the latter is also an isomorphism and hence that $L$ is relatively pin.  Using Zapolsky's machinery, this allows us to take $R = \Z$.

The Auroux--Kontsevich--Seidel criterion (\cite[Proposition 6.8]{AurouxMSandTduality}, \cite[Lemma 2.7]{SheridanFano}) now tells us that $\CO(2(n+1)h) = 0$, so $H^1(L; \Z)$ (which coincides with $HF^1(L, L; \Z)$) is $2(n+1)$-torsion and therefore vanishes ($H^1$ is always torsion-free).  A topological argument then shows that $i^*h$ has order $2$ in $H^2(L; \Z)$, so $\CO(2h)=0$.  Hence, by the quantum module action of $h$, all intermediate compactly-supported cohomology groups of $L'$ are $2$-torsion and $2$-periodic. Letting $L'$ range through the covers of $L$ corresponding to cyclic subgroups of $\pi_1(L)$ yields the result.

\subsection{Structure of the paper}

In \cref{section soft concepts} we review the Maslov index and  monotonicity.  \Cref{section Floer review} then gives a summary of Zapolsky's canonical pearl complex, including: its algebraic structures (\cref{subsection algebraic structures}); their relation to classical operations (\cref{classical comparison}); orientations and the relevance of relative pin structures (\cref{RelPin}); local systems (\cref{TwistedCoeffs}); the worked example of $\R\P^n$ (\cref{RPnExample}), which we compute in a different way from Zapolsky; and an outline of Damian's methods (\cref{subsection Damian}).  Finally, \cref{section proof} contains the full proof of \cref{Theorem1}.

\subsection{Acknowledgements}

We would like to thank Jonny Evans for helpful feedback, an anonymous referee for valuable suggestions and comments, and University College London where both of us were based whilst writing this paper.
MK was supported by EPSRC grant [EP/L015234/1], 
the EPSRC Centre for Doctoral Training in Geometry and Number Theory (The London School of Geometry and Number Theory), University College London.
JS was supported by EPSRC grant [EP/P02095X/1].

%====================================================================
\section{Soft concepts}\label{section soft concepts}
We begin by recalling some general facts about the Maslov class, the minimal Maslov number and monotonicity. For completeness we state the definitions and observations in their most general form, but for the purposes of the rest of this paper we will only use the special case \cref{lemma: the restriction of the hyperplane class to a lag in cpn}, so the reader familiar with the concepts is invited to skip the interlude. In this section all homology and cohomology groups are considered with $\Z$ coefficients, unless explicitly specified otherwise. 

Let $(X, J)$ be an almost complex manifold of real dimension $2n$ and $L \subset X$ a properly embedded totally real submanifold of dimension $n$. The bundle $\Lambda^n_{\R}TL$ is naturally a rank $1$ real subbundle of $\at{\Lambda^n_{\C}TX}{L}$, so the bundle pair $(\Lambda^n_\C TX, \Lambda^n_\R TL)$ over $(X, L)$ is classified by a map
\[
\phi : (X, L) \rightarrow (B\mathrm{U}(1), B(\Z/2)).
\]
One can view the pair $(B\mathrm{U}(1), B(\Z/2))$ %either abstractly as the projection
%\[
%B(\Z/2) \simeq E(\Z/2) \times_{\Z/2} E\mathrm{U}(1) \rightarrow E\mathrm{U}(1)/\mathrm{U}(1) \simeq B\mathrm{U}(1),
%\]
%where $\Z/2$ acts on $E\mathrm{U}(1)$ via the subgroup $\{\pm 1\}$, or more concretely
as
\[
B(\Z/2) \cong \R\P^\infty = \mathrm{Gr}_\R(1, \R^\infty) \xrightarrow{\ \otimes \C \ } \mathrm{Gr}_\C(1, \C^\infty) = \C\P^\infty \cong B\mathrm{U}(1).
\]
The long exact sequence for the pair shows that $H^2(B\mathrm{U}(1), B(\Z/2))$ is isomorphic to $\Z$, generated by a relative characteristic class which maps to $2c_1$ in $H^2(B\mathrm{U}(1))$.  This generator is called the \emph{Maslov class}, denoted by $\mu$, and its pullback via $\phi$ is the Maslov class of $L$, denoted by $\mu_L \in H^2(X, L)$.  If $j^* \colon H^2(X, L) \to H^2(X)$ is the natural restriction map, then it is clear from the above description that one has
\begin{equation}\label{eq Viterbo}
j^*(\mu_L) = 2c_1(X).
\end{equation}
We will write $I_{\mu_L} \colon H_2(X, L) \to \Z$ and $I_{c_1}\colon H_2(X) \to \Z$ for the group homomorphisms given by pairing with $\mu_L$ and $c_1$ respectively. 

\begin{rmk}\label{remarkMuvsW1}
Note that the long exact sequence of the pair $(B\mathrm{U}(1), B(\Z/2))$ with $\Z/2$-coefficients shows that the mod $2$ reduction of $\mu_L$ equals the image of the first Stiefel--Whitney class of $TL$ under the co-boundary map $H^1(L;\Z/2) \to H^2(X,L;\Z/2)$. In particular, for any class $A \in H_2(X, L)$, the parity of $I_{\mu_L}(A)$ is determined by whether the pairing of $w_1(TL)$ with $\partial A$ vanishes. Thus, if $L$ is orientable then $I_{\mu_L}$ has image in $2\Z$ and, conversely, if $I_{\mu_L}(H_2(X, L)) \le 2\Z$ and the boundary map $H_2(X, L) \to H_1(L)$ is surjective (e.g.~if $H_1(X) = 0$), then $L$ is orientable.
\end{rmk}

Now let $H_2^D(X, L)$ and $H_2^S(X)$ denote the images of the Hurewicz homomorphisms
\[
\pi_2(X, L) \to H_2(X, L) \text{\quad and \quad} \pi_2(X) \to H_2(X)
\]
and let $j \colon H_2(X) \to H_2(X, L)$ be the natural map. Define the integers $N_L^{\pi}$, $N_L^H$, $N_X^{\pi}$ and $N_X^H$ to be the non-negative generators of the $\Z$-subgroups $I_{\mu_L}(H_2^D(X, L))$, $I_{\mu_L}(H_2(X, L))$, $I_{c_1}(H_2^S(X))$, $I_{c_1}(H_2(X))$, respectively. Using \eqref{eq Viterbo} and the fact that $j(H_2^S(X)) \le H_2^D(X, L)$, it is easy to see that there exist non-negative integers $k_L, k_X, m_{\pi}, m_H$ such that:
\begin{equation}
\label{equation minimal numbers}
N_L^{\pi} = k_L N_L^H, \quad N_X^{\pi} = k_X N_X^H, \quad 2N_X^{\pi} = m_{\pi}N_L^{\pi}, \quad 
2N_X^H = m_H N_L^H.
\end{equation}
Observe that if $N_L^H \neq 0$ (e.g. if $N_X^H \neq 0$), then one has the identity 
\begin{equation}\label{eq relation between kl, km, mpi and mh}
k_L m_{\pi} = k_X m_H.
\end{equation}

We note the following result for later:

\begin{lem}\label{lemma: the main soft observation lemma}
Suppose that $N_L^H \neq 0$, $H^1(L) = 0$, and $H^2(X)$ is isomorphic to $\Z$, generated by some class $h$. Then the restriction of $h$ to $H^2(L)$ has order $m_H$.
\end{lem}
\begin{proof}
%Since $H_1(L)$ is finite, the universal coefficients theorem tells us that $H^1(L) = 0$. Then 
The long exact sequence in cohomology for the pair $(X, L)$ yields the exact sequence
\[
\xymatrix{
H^1(L) \ar[r]\ar@{=}[d]& H^2(X, L) \ar[r]\ar@{=}[d]& H^2(X) \ar[r]^{i^*}\ar@{=}[d]& H^2(L)\ar@{=}[d]\\
0 \ar[r]& H^2(X, L) \ar[r]^{j^*}& \Z\langle h \rangle \ar[r]^{i^*}& H^2(L),}
\]
where $i : L \rightarrow X$ is the inclusion.  This tells us that $H^2(X, L)$ injects into $\Z\langle h \rangle$ and so is freely generated by some class $g \in H^2(X, L)$, which is non-zero since $N_L^H \neq 0$.  By the universal coefficients theorem there exists a class $u \in H_2(X, L)$ with which $g$ pairs to $1$, and hence $\mu_L = N_L^H g$.  The same argument shows that $c_1 = N_X^H h$. Applying $j^*$ to the identity $\mu_L = N_L^H g$ and using \eqref{eq Viterbo} we then get $2 N_X^H h = N_L^H j^*(g)$ and hence $j^*(g) = m_H \,h$. By exactness of the above diagram it follows that $i^*(h)$ has order $m_H$ in $H^2(L)$.
\end{proof}

Consider now the case when $(X, \omega)$ is symplectic and $L$ is a Lagrangian submanifold. Then $L$ is totally real with respect to any almost complex structure compatible with the symplectic form. 
 In this setting we also have homomorphisms $I_{\omega} \colon H_2(X) \to \R$,  $I_{\omega, L} \colon H_2(X, L) \to \R$ given by integration of the symplectic form. The manifold $(X, \omega)$ is called \emph{monotone} if there exists a positive constant $\lambda$ such that 
\[\at{I_{\omega}}{H_2^S(X)} = 2 \lambda \at{I_{c_1}}{H_2^S(X)}.\]
For example, $(\C\P^n, \omega_{FS})$ is monotone with $\lambda = \pi/2(n + 1)$ when the Fubini-Study form is normalised so that a line has area $\pi$. 
In turn, the Lagrangian submanifold $L$ is called \emph{monotone} if
\[\at{I_{\omega, L}}{H_2^D(X, L)} = \lambda' \at{I_{\mu_L}}{H_2^D(X, L)}\]
for some positive constant $\lambda'$.  Note that if $\at{I_{c_1}}{H_2^S(X)} \neq 0$ then \eqref{eq Viterbo} implies that a monotone Lagrangian can only exist if $X$ itself is monotone and $\lambda'$ coincides with $\lambda$. 

In the literature on holomorphic curves, the numbers $N_X^{\pi}$ and $N_L^{\pi}$ are usually the ones referred to as the \emph{minimal Chern number} of $X$ and the \emph{minimal Maslov number} of $L$, respectively. This can potentially cause confusion since these numbers are not the same as $N_X^H$ and $N_L^H$ in general. However, if $X$ is simply connected (for example, if it is a projective Fano variety---see \cite{KollarMiyaokaMori} and \cite[Theorem 3.5]{Campana}), then these numbers coincide. Indeed, we have the commutative diagram
\[
\xymatrix{
\pi_2(X) \ar[d] \ar[r]& \pi_2(X, L) \ar[d]\ar[r]& \pi_1(L) \ar[d]\ar[r]& 1 \ar[d]\\
H_2(X) \ar[r] & H_2(X, L) \ar[r]& H_1(L) \ar[r]& 0
}
\]
in which the third vertical arrow is a surjection by Hurewicz.  If $X$ is simply connected then the first vertical arrow is also a surjection, again by Hurewicz, so $N_X^\pi = N_X^H$. A diagram chase in the spirit of the $5$-lemma (or alternatively, noticing that $\pi_1(X,L) = 0$ and applying the relative Hurewicz theorem) then shows that the second vertical arrow must also be a surjection, from which we deduce that $N_L^{\pi} = N_L^H$.  In this case there is therefore no ambiguity, and we denote the common values simply by $N_X$ and $N_L$ respectively.

Consider again the example of $X = \C\P^n$, with $L \subset \C\P^n$ a totally real submanifold. We then have $N_{\C\P^n} = n + 1$ and by \eqref{equation minimal numbers} $N_L$ is non-zero and divides $2(n + 1)$. As a corollary of \cref{lemma: the main soft observation lemma} we immediately obtain:
\begin{lem}\label{lemma: the restriction of the hyperplane class to a lag in cpn}
If $L \subset \C\P^n$ is a totally real submanifold with $H^1(L) = 0$ and minimal Maslov number $N_L$
% and let $u \in H_2(\C\P^n, L)$ be a class with $I_{\mu_L}(u) = N_L$. If $\partial u$ has finite order in $H_1(L)$ then that order is divisible by $\frac{2(n + 1)}{N_L}$ and so is the size of the torsion subgroup of $H_1(L)$. If further $H_1(L)$ is finite, 
then the restriction of the hyperplane class $h \in H^2(\C\P^n)$ to $L$ has order $2(n + 1)/N_L$ in $H^2(L)$.
\end{lem}

%====================================================================
\section{Floer theory review}
\label{section Floer review}

\subsection{The canonical pearl complex}
\label{CanonicalComplex}

Our argument for \cref{Theorem1} is based on consideration of the self-Floer theory of $L$.  In particular, we employ Zapolsky's \emph{canonical pearl complex} which we now review briefly; see \cite{Zapolsky} for further details.  Let $(X, \omega)$ be a symplectic manifold which is either closed or convex at infinity, and let $L \subset X$ be a closed, connected, monotone Lagrangian submanifold of minimal Maslov number $N_L^{\pi}$ at least $2$. We consider the following condition \cite[Section 1.2]{Zapolsky}: 
\begin{center}
\begin{tabular}{m{0.95\textwidth}}
\emph{Assumption (O):} For some point (or, equivalently, all points) $q$ in $L$, the second Stiefel--Whitney class $w_2(TL)$ vanishes on the image of $\pi_3(X, L, q)$ in $\pi_2(L, q)$ under the boundary map.
\end{tabular}
\end{center}
As explained in \cref{RelPin} below, this is implied by $L$ being relatively pin. Fix a ground ring $R$.  This must have characteristic $2$ unless $L$ satisfies assumption (O), in which case it is arbitrary.

Fix a generic choice of $\omega$-compatible almost complex structure $J$ on $X$. For each point $q$ in $L$ and each class $A$ in $\pi_2(X, L, q)$, Zapolsky considers the family $D_A$ of linear Cauchy--Riemann operators over the space of based discs in class $A$.  We restrict these operators to vector fields vanishing at the base point and denote the resulting family by $D_A \# 0$.  By assumption (O), the index bundle of $D_A \# 0$ is orientable (see \cref{RelPin} below) and we define $\mathscr{C}(q, A)$ to be the free $R$-module of rank $1$ generated by its two orientations, modulo the relation that they sum to zero.  Taking the direct sum of these modules over $A$ we obtain a module $\mathscr{C}^*_q$, graded by the Maslov index: 
\[\mathscr{C}^r_q \coloneqq \bigoplus_{\substack{A \in \pi_2(X, L, q)\\ I_{\mu_L}(A) = r}} \mathscr{C}(q, A).\] As $q$ varies these modules fit together to form a local system over $L$ which we denote by $\mathscr{C}^*$.  Note that $\mathscr{C}^0$ contains a copy of the trivial local system given by $\mathscr{C}(q, 0)$ inside each fibre $\mathscr{C}^0_q$; we denote this by $\mathscr{C}^\mathrm{triv}$.

%Now choose a Morse function $f$ on $L$ and a metric $g$ such that the pair $(f, g)$ is Morse--Smale. For each critical point $q \in \mathrm{Crit}(f)$ and $A \in \pi_2(X, L, q)$ consider again the family of operators $D_A$ but this time restrict them to vector fields which at $1$ are tangent to the descending manifold of $f$ at $q$. By assumption (O) the index bundle of this family is again orientable (its orientation differs from a chosen orientation of $\mathrm{det}(D_A\# 0)$ by a choice of orientation of the descending manifold of $q$) and let $C(q, A)$ denote the corresponding free $R$-module.  Zapolsky's complex $CF^{*, *}_\mathrm{Zap} (L, L; R)$ is given by
%\[
%CF^{r, s}_\mathrm{Zap} (L, L; R) = \bigoplus_{\substack{q \in \mathrm{Crit}(f) \\ |q| = s}}
%\bigoplus{\substack{A \in \pi_2(X, L, q) \\ I_{\mu_L}(A) = r}} C(q, A),
%\]
%where $|q|$ denotes the Morse index of the critical point $q$.  The differential, of total degree $1$, counts \emph{pearly trajectories}, meaning upwards Morse flow lines which may be interrupted by the boundaries of $J$-holomorphic discs in $X$ bounded by $L$.  More precisely, for each such trajectory $u$ from $q$ to $q'$ and each class $A$ in $\pi_2(X, L, q)$ Zapolsky defines a class $A \# u$ in $\pi_2(X, L, q')$ and an isomorphism
Now choose a Morse function $f$ on $L$ and a metric $g$ such that the pair $(f, g)$ is Morse--Smale. For each critical point $q \in \mathrm{Crit}(f)$ let $C(q)$ denote the rank $1$ free $R$-module generated by the orientations of the descending manifold of $q$ (modulo summing to zero, as usual; we won't keep repeating this). Zapolsky's complex $CF^{*, *}_\mathrm{Zap} (L, L; R)$ is given by \cite[Section 4.2.1]{Zapolsky}
\[
CF^{r, s}_\mathrm{Zap} (L, L; R) = \bigoplus_{\substack{q \in \mathrm{Crit}(f) \\ |q| = s}}\mathscr{C}^r_q \otimes_R C(q),
\]
where $|q|$ denotes the Morse index of the critical point $q$.  The differential, of total degree $1$, counts rigid \emph{pearly trajectories}, meaning upwards Morse flow lines which may be interrupted by the boundaries of $J$-holomorphic discs in $X$ bounded by $L$.  More precisely, for each critical point $q \in \mathrm{Crit}(f)$ and each class $A \in \pi_2(X, L, q)$ one considers the module $C(q, A) \coloneqq \mathscr{C}(q, A) \otimes_R C(q)$. For each pearly trajectory $u$ from $q$ to $q'$ Zapolsky defines a class $A \# u$ in $\pi_2(X, L, q')$ and an isomorphism
\[
C(u, A) : C(q, A) \rightarrow C(q', A \# u).
\]
The differential $\partial$ is then the sum of all these maps $C(u, A)$.  See \cite[Section 4.2.2]{Zapolsky} for full details.  We denote the resulting cohomology by $HF^*_{\mathrm{Zap}}(L, L; R)$.

Crucially, $\partial$ has non-negative degree with respect to the $r$-grading: in fact it decomposes as $\partial_0 + \partial_1 + \dots$ where $\partial_j$ has bigrading $(jN_L^{\pi}, 1-jN_L^{\pi})$.  Filtering by this grading we therefore obtain

\begin{prop}[{\cite[Theorem 4.17]{Zapolsky}}]
There is a spectral sequence---the Oh (or Biran) spectral sequence---which starts at the (Morse) cohomology of $L$ with coefficients in the local system $\mathscr{C}^*$, and which converges to $HF^*_{\mathrm{Zap}}(L, L; R)$.
\end{prop}

\begin{rmk}
The monodromy of $\mathscr{C}^*$ has two contributions: one from the local system $\pi_2(X, L, q) \mapsto q$ over $L$, and one from the index bundles of the operators $D_A \# 0$.  In \cite[Theorem 4.17]{Zapolsky} the former is considered, but in general the latter is necessary too, as in the computation in \cref{RPnExample}.
\end{rmk}

Since the $r$-grading is concentrated in $N_L^{\pi} \Z$, if we laid out the spectral sequence in the standard way then only one in every $N_L^{\pi}$ columns and pages would be interesting.  We therefore squash it up so that the $E_1$-page is given by
\[
E_1^{a, b} = H^{a + b - N_L^{\pi}a}(L; \mathscr{C}^{a})
\]
and the differential on the $E_j$-page is $\partial_j$, acting from the $(a, b)$-entry to the $(a+j, b-j+1)$-entry.  Note that the zeroth column of $E_1$ contains a copy of the usual cohomology of $L$ over $R$, corresponding to $\mathscr{C}^\mathrm{triv} \subset \mathscr{C}^0$.

\subsection{Algebraic structures}
\label{subsection algebraic structures}

The Floer product can be defined on the canonical pearl complex by counting $Y$-shaped pearly trajectories \cite[Section 4.2.3]{Zapolsky}.  As with the differential, one has to keep track of homotopy classes attached to generators of the complex, and define appropriate maps between the modules $C(q, A)$.  The product has a unit $1_L$ coming from the summand $C(\text{Morse min}, 0)$ of $CF^{0,0}_\mathrm{Zap}(L, L; R)$ \cite[Section 4.2.4]{Zapolsky}.

Zapolsky similarly defines a canonical pearl-type complex $QC^{*,*}_\mathrm{Zap}(X; R)$ for the quantum cohomology of $X$ \cite[Section 4.5.1]{Zapolsky}.  This time, he constructs a module $\mathscr{C}(x, B)$ for each point $x$ in $X$ and for each class $B$ in $\pi_2(X, x)$, and since the relevant Cauchy--Riemann operators are all canonically oriented there is no need for an analogue of assumption (O).  These modules assemble into a local system on $X$ (isomorphic to $R[\pi_2(X)]$), and the boundary operator and product are defined in an analogous way to the Lagrangian case above \cite[Sections 4.5.1--4.5.2]{Zapolsky}.  Again there is a canonical trivial subsystem of the local system, given by $\mathscr{C}(x, 0)$ in each fibre.

The closed--open string map (or quantum module structure) carries over to this setting to give

\begin{prop}[{\cite[Section 4.5.4]{Zapolsky}}]
There is a unital $R$-algebra homomorphism
\[
\CO : QH^*_\mathrm{Zap}(X; R) \rightarrow HF^*_\mathrm{Zap}(L, L; R).
\]
\end{prop}

%------------------------------------------------------------------------------------------------------
\subsection{Comparison with classical operations}
\label{classical comparison}

There are obvious inclusions of the Morse complexes of $X$ and $L$ into $QC^{0,*}_\mathrm{Zap}(X; R)$ and $CF^{0,*}_\mathrm{Zap}(L, L; R)$, sending a critical point $q$ (plus an $R$-orientation of its descending manifold) to the corresponding generator of $C(q)$ tensored with the canonical generator of $\mathscr{C}(q, 0)$. For quantum cohomology this gives the usual inclusion
\[
H^*(X; R) \rightarrow QH^*_\mathrm{Zap}(X; R)
\]
of $R$-modules, whilst for Lagrangian Floer cohomology we only obtain a map
\[
H^{<N_L^{\pi}-1}(L; R) \rightarrow HF^*_\mathrm{Zap}(L, L; R),
\]
which can be viewed as the inclusion of $H^{<N_L^{\pi}-1}(L; R)$ in the zeroth column of the first page of the Oh spectral sequence.  We refer to both of these as ``PSS maps'' because, when one uses a Hamiltonian model for the right-hand sides, they coincide with the usual morphisms constructed by Piunikhin-Salamon-Schwarz (\cite{PSS}) and Albers (\cite{Albers}) respectively.

In the Lagrangian case the map actually extends to the kernel of the spectral sequence differential $\partial_1 : H^{\leq N_L^{\pi}-1}(L; R) \rightarrow H^{\leq 0}(L; R)$ and we denote this extended domain by $H^\mathrm{PSS}(L; R)$.  The reason is that this kernel maps to the $E_2$ page, from which point onwards it lies in the kernel of each page differential for degree reasons.  Hence it maps to $E_\infty \cong \operatorname{gr} HF^*_\mathrm{Zap}(L, L; R)$, where it sits in the top piece of the associated grading (which is actually the bottom right-hand end of each descending diagonal on $E_\infty$) and therefore maps from $\operatorname{gr} HF^*_\mathrm{Zap}(L, L; R)$ to $HF^*_\mathrm{Zap}(L, L; R)$ itself.  In simpler terms, a Morse cocycle on $L$ of degree $\leq N_L^\pi -1$ is Floer-closed if and only if it lies in the kernel of $\partial_1$, and is Floer-exact if it is Morse-exact.

These PSS maps preserve the units $1_X$ and $1_L$, but do not in general respect the product structures unless the total degree of the classes being multiplied is less than the minimal Chern number $N_X^{\pi}$ or the minimal Maslov number $N_L^{\pi}$, for $QH^*$ and $HF^*$ respectively.  Similarly, $\CO$ is related to the classical restriction map $i^* : H^*(X; R) \rightarrow H^*(L; R)$ by the following commutative diagram:
\begin{equation}\label{CODiagram}
\begin{tikzcd}
H^{<N^\pi_L}(X; R) \arrow{r}{i^*} \arrow{d}[swap]{\mathrm{PSS}} & H^\mathrm{PSS}(L; R) \arrow{d}{\mathrm{PSS}}
\\ QH^*_\mathrm{Zap}(X; R) \arrow{r}{\CO} & HF^*_\mathrm{Zap}(L, L; R)
\end{tikzcd}
\end{equation}
Note that the image of $H^{<N^\pi_L}(X; R)$ under $i^*$ is contained in $H^\mathrm{PSS}(L; R)$ since at chain level $i^*$ coincides with $\CO$ on $C^{< N^\pi_L}(X)$, and $\CO$ is a chain map with respect to the ordinary differential on $C^*(X)$ and the pearl differential on $C^*(L)$.

At this point we can already deduce the main workhorse of the current paper. 
%\begin{lem}
%Let $X$ be a simply connected, monotone symplectic manifold and let $L \subset X$ be a closed, monotone Lagrangian submanifold with $N_L \ge 3$. Suppose further that $H^1(L;\Z) = 0$ and $H^2(X;\Z) = \Z\langle h \rangle$ for some class $h$ which is invertible in quantum cohomology. Assume that $L$ satisfies assumption (O). Then $HF^*_\mathrm{Zap}(L, L;R)$ is $(2 N_X/N_L)$-torsion.
%\end{lem}
%\textcolor{red}{I am not sure how useful this level of generality is. For example, Givental's paper "Equivariant GW invariants", Corollary 9.3 shows that $\C\P^n$ is the only Fano complete intersection which satisfies the above requirements (unless a stronger relation than the one stated in Corollary 9.3 holds, but I think Givental would have written that instead then). If we don't find any other examples, maybe it's better to just leave the statement just for $CP^n$?}

\begin{lem}\label{lemma: torsion of Floer cohomology}
Let $L \subset \C\P^n$ be a closed, monotone Lagrangian with $N_L \ge 3$ and satisfying assumption (O). Suppose that $H^1(L;\Z) = 0$. Then $HF^*_\mathrm{Zap}(L, L;R)$ is $(2(n + 1)/N_L)$-torsion.
\end{lem}

\begin{rmk}
If $R$ has characteristic $2$ we do not need to assume that $L$ satisfies assumption (O). In this case the result is only interesting if $2 (n+1)/N_L$ is odd, in which case it tells us that $HF^*_\mathrm{Zap}(L,L;\Z/2)$ vanishes.
\end{rmk}

\begin{proof}
The manifold $\C\P^n$ is simply connected, has minimal Chern number $2(n+1)$, and has $\pi_2(\C\P^n, q)$ isomorphic to $\Z$ for any base point $q$. This means that the local system $\Z[\pi_2(\C\P^n)]$ on which Zapolsky's quantum cohomology pearl complex lives is simply the constant sheaf with fibre the Novikov ring $\Z[T^{\pm 2}]$, where $T$ has degree $n+1$ (it might seem more natural to work with the Novikov ring $\Z[T^{\pm 1}]$, with $T$ assigned degree $2(n+1)$, but then we would have to introduce a square root of $T$ when we came to discuss the Floer cohomology of $L$).  Zapolsky's construction therefore yields the standard quantum cohomology ring of $\C\P^n$ \cite[Example 8.1.6]{SmallMcDuffSalamon}, namely
\[
QH^*_\mathrm{Zap} (\C\P^n; \Z) \cong \Z[h, T^{\pm 2}]/(h^{n+1}-T^2),
\]
where $h$ is (the PSS image of) the hyperplane class. 
%Note that $h$ is invertible and of degree $2$, so by the quantum module action $HF^*_\mathrm{Zap}(L, L; \Z)$ must be $2$-periodic in its $\Z$-grading. 
Our assumption $N_L \ge 3$ implies that $i^*h$ lies in $H^\mathrm{PSS}(L;R)$, and by the commutativity of \eqref{CODiagram} we have 
\[\CO(h) = \mathrm{PSS}(i^* h) \in HF^*_\mathrm{Zap}(L, L;R).\]
Since $H^1(L;\Z) = 0$, \cref{lemma: the restriction of the hyperplane class to a lag in cpn} implies that $(2(n+1)/N_L)i^*h = 0$ and so $\CO((2(n+1)/N_L)h) = 0$. Since $h$ is invertible, it follows that $HF^*_\mathrm{Zap}(L, L;R)$ is $(2(n + 1)/N_L)$-torsion.
\end{proof}

%---------------------------------------------------------------------------------------------- 
\subsection{Orientation and relative pin structures}
\label{RelPin}

In this subsection we explain how:
\begin{enumerate}[(i)]
\item\label{itm1} assumption (O) allows the definition of the local system $\mathscr{C}^*$ \cite[Lemma 4.1]{Zapolsky}
\item\label{itm2} the monodromy of the local system $\mathscr{C}^*$ can be computed (\cref{lem:explicit monodromy} below)
\item\label{itm3} the existence of a relative pin structure implies assumption (O) \cite[Remark 7.1]{Zapolsky}
\item\label{itm4} the choice of such a structure allows one to recover a more standard version of Floer theory \cite[Sections 7.2--7.4]{Zapolsky}.
\end{enumerate}
For our applications the reader can happily skip to \cref{prop Quotient Properties} if they are willing to take (\ref{itm1}), (\ref{itm3}) and (\ref{itm4}) as a black box from \cite{Zapolsky}.  The only place we explicitly use (\ref{itm2}) is in the example computation in \cref{RPnExample}.  To fix notation, let $D$ denote the closed unit disc in $\C$, and $S^1 = \partial D$ its boundary.  Let $\langle \cdot, \cdot \rangle_Z$ be the mod $2$ pairing between homology and cohomology on a space $Z$.

We begin by having a closer look at how the local system $\mathscr{C}^*$ is constructed. 
%Let $\widetilde{L}$ be the universal cover of $L$ and 
Let $\pi \colon \overline{L} \to L$ denote the cover of $L$ with fibres $\pi^{-1}(q) = \pi_2(X,L,q)$. Consider the space $C^{\infty}_{X,L} \coloneqq C^{\infty}((D,\partial D), (X,L))$. It fibres over $\overline{L}$ via a map $\psi \colon C^{\infty}_{X,L} \to \overline{L}$, which associates to every disc $u$ its relative homotopy class based at $u(1)$. It is important that the fibres of $\psi$ are connected and the evaluation map $\mathrm{ev}_1 \colon C^{\infty}_{X,L} \to L$ factors as $\mathrm{ev}_1 = \pi \circ \psi$.

Now consider the family of Fredholm operators $D_\bullet$ over $C^{\infty}_{X,L}$, where for each disc $u$ the operator
\[
D_u \colon W^{1, p}((D, \partial D), (u^*TX, u^*TL)) \to L^p(D,\overline{\Hom}_{\C}((TD, i), (u^*TX, u^*J)))
\]
is the linearisation of $\bar{\partial}_J$ (for some fixed and irrelevant choice of connection on $X$). The determinant lines of these operators give rise to a line bundle $\det(D_\bullet)$ on $C^{\infty}_{X,L}$, whose first Stiefel--Whitney class has been computed by Seidel in \cite[Lemma 11.7]{SeidelBook} (there is a slight subtlety here: the bundle is not canonically topologised, but rather there exists an uncountable family of suitable choices, as described in \cite{ZingerDeterminantLine}; this is taken care of by Zapolsky and we will not dwell on it). Namely, if $v \colon S^1 \times (D, \partial D) \to (X,L)$ is a loop in $C^{\infty}_{X,L}$ with 
$v_*[\{1\}\times D] = A \in \pi_2(X,L, v(1, 1))$, then 
\begin{equation}\label{eqSeidelFormula}
\langle w_1(\det(D_\bullet)), v \rangle_{C^{\infty}_{X, L}} = 
\langle w_2(TL), v_*[S^1 \times \partial D] \rangle_L 
+ (I_{\mu_L}(A) - 1)\langle w_1(TL), v_*[S^1 \times \{1\}]\rangle_L.
\end{equation}

\begin{rmk}
Seidel's setup and notation are slightly different.  To translate into his language we first need to trivialise $v^*TX$, identifying all fibres with a standard symplectic vector space $V$.  At each time $t$, we obtain a loop $v|_{\{t\} \times \partial D}$ in the Lagrangian Grassmannian $\mathrm{Gr}(V)$ of $V$, and hence a point in its free loop space $\mathscr{L} \mathrm{Gr}(V)$.  As $t$ varies in $S^1$, these points sweep a $1$-chain in $\mathscr{L} \mathrm{Gr}(V)$, and we denote this by $\sigma$.  Seidel introduces operators
\[
T \colon H^{k+1}(\mathrm{Gr}(V); \Z) \to H^k(\mathscr{L} \mathrm{Gr}(V);\Z) \text{\quad and \quad} U\colon H^k(\mathrm{Gr}(V); \Z) \to H^k(\mathscr{L} \mathrm{Gr}(V);\Z)
\]
on cohomology, whose duals take a $k$-chain on $\mathscr{L} \mathrm{Gr}(V)$ and output, respectively, the $(k+1)$- and $k$-chains on $\mathrm{Gr}(V)$ swept by the $k$-chain of whole loops and by the $k$-chain of initial points of the loops.  His formula is given in terms of these operators as
\[
\langle w_1(\det(D_\bullet)), \sigma\rangle_{\mathscr{L} \mathrm{Gr}(V)} = 
\langle T(w_2),\sigma \rangle_{\mathscr{L} \mathrm{Gr}(V)} + 
\langle (T(\mu) - 1)\smile U(\mu), \sigma \rangle_{\mathscr{L} \mathrm{Gr}(V)}.
\]
The right-hand side can be shown to be independent of the initial choice of trivialisation (changing trivialisation does not affect $T(w_2)$ and $T(\mu)$ and preserves the parity of $U(\mu)$).  To derive \eqref{eqSeidelFormula}, simply use the fact that $w_1$ of the tautological bundle over $\mathrm{Gr}(V)$ equals the mod $2$ reduction of $\mu \in H^1(\mathrm{Gr}(V);\Z)$.
\end{rmk}

Recall however that for the construction of the local system $\mathscr{C}^*$ we are actually interested in the family of operators $D_\bullet\# 0$, whose determinant line bundle $\det(D_\bullet \# 0)$ over $C^{\infty}_{X,L}$ is canonically (up to a positive real multiple) isomorphic to $\det(D_\bullet) \otimes \mathrm{ev}_1^*(\det(TL))$.
So we have:
\begin{equation}\label{MonodromySign}
\langle w_1(\det(D_\bullet\# 0)), v \rangle_{C^{\infty}_{X, L}} = 
\langle w_2(TL), v_*[S^1 \times \partial D] \rangle_L 
+ I_{\mu_L}(A)\langle w_1(TL), v_*[S^1 \times \{1\}]\rangle_L.
\end{equation}

Now suppose that the loop $v$ is contained entirely in a fibre $\psi^{-1}(q, A)$. Then $v(S^1 \times \{1\}) = q$ so the second term in \eqref{MonodromySign} vanishes, and assumption (O) implies the vanishing of the first term (indeed, assumption (O) is equivalent to the vanishing of the first term for all loops whose image under $\ev_1$ is constant). Thus the restriction $\at{\det(D_\bullet \# 0)}{\psi^{-1}(q, A)}$ is orientable for all $q$ and $A$ if and only if assumption (O) holds. Assume that this is the case and define $\mathscr{C}(q, A)$ to be the free rank $1$ $R$-module generated by its orientations (this makes sense, meaning that there are only two possible orientations, which differ by sign, because $\psi^{-1}(q, A)$ is connected). These modules now define a local system $\overline{\mathscr{C}}$ over $\overline{L}$ and \eqref{MonodromySign} in principle allows us to also compute the monodromy of $\overline{\mathscr{C}}$. By taking the direct sums of the fibres of $\overline{\mathscr{C}}$ over points in the fibres of $\pi \colon \overline{L} \to L$---i.e.~by pushing forward $\overline{\mathscr{C}}$ by $\pi$---we obtain the local system $\mathscr{C}^*$ on $L$. 

%For loops along which $\gamma(t) = u(t, 1)$ is constant, the second term in \eqref{MonodromySign} vanishes.  Assumption (O) is equivalent to the first term also vanishing for all such loops, which is what is needed in order for the bundles used to define the $C(q, A)$ to be orientable. 

Suppose now that $L$ is relatively pin, meaning that $w_2(TL)$ or $w_2(TL) + w_1(TL)^2$ is in the image of the restriction map $i^* : H^2(X; \Z/2) \rightarrow H^2(L; \Z/2)$. The positive consequences of this assumption are twofold: first, it implies that $L$ satisfies assumption (O) and so the local system $\overline{\mathscr{C}}$ is well defined; second, it significantly simplifies the computation of the monodromy of $\overline{\mathscr{C}}$. These follow because the first term in \eqref{MonodromySign} is zero for \emph{all} loops of discs, not just the ones contained in a single fibre of $\psi$. To see this, suppose first that $w_2(TL) = i^* b$ for some background class $b$ in $H^2(X; \Z/2)$.  We then have
\[
\big\langle w_2(TL), v_*[S^1 \times \partial D] \big\rangle_L = \big\langle b, \partial(v_*[S^1 \times D]) \big\rangle_X = 0.
\]
If, on the other hand, we have $w_2(TL) + w_1(TL)^2 = i^*b$ then the same argument shows that
\[
\big\langle w_2(TL), v_*[S^1 \times \partial D] \big\rangle_L = \big\langle (v^*w_1(TL))^2, [S^1 \times \partial D] \big\rangle_{S^1 \times \partial D},
\]
and the right-hand side vanishes since the squaring map $H^1(T^2; \Z/2) \rightarrow H^2(T^2; \Z/2)$ on the torus $T^2 = S^1 \times \partial D$ is zero.  We deduce:

\begin{lem}
\label{lem:explicit monodromy}
If $L$ is relatively pin and $\gamma \in \pi_1(L, q)$ fixes a class $A \in \pi_2(X,L,q)$ (i.e. $\gamma$ lifts to a loop in $\overline{L}$), then the monodromy action of $\gamma$ on $\mathscr{C}(q, A)$ is trivial if the Maslov index of $A$ is even or if $\gamma$ is an orientation-preserving loop in $L$, and is multiplication by $-1$ otherwise.
\end{lem}

%assumption (O) is automatically satisfied, because the first term in \eqref{MonodromySign} is zero for \emph{all} loops of discs, not just the ones for which the value at $1$ is fixed.  To see this, suppose first that $w_2(TL) = i^* b$ for some background class $b$ in $H^2(X; \Z/2)$.  We then have
%\[
%\big\langle w_2(TL), u_*[S^1 \times \partial D] \big\rangle_L = \big\langle b, \partial(u_*[S^1 \times D]) %\big\rangle_X = 0.
%\]
%If, on the other hand, we have $w_2(TL) + w_1(TL)^2 = i^*b$ then the same argument shows that
%\[
%\big\langle w_2(TL), u_*[S^1 \times \partial D] \big\rangle_L = \big\langle (u^*w_1(TL))^2, [S^1 \times %\partial D] \big\rangle_{S^1 \times \partial D},
%\]
%and the right-hand side vanishes since the squaring map $H^1(T^2; \Z/2) \rightarrow H^2(T^2; \Z/2)$ on the torus $T^2 = S^1 \times \partial D$ is zero.

%\textcolor{red}{Observe further that when $L$ is relatively pin, we can easily compute the monodromy of the local system $\mathscr{C}^*}$ as it only depends on the action of $\pi_1(L ,q)$ on $\pi_2(X,L,q)$ and the second term in \eqref{MonodromySign}. Indeed, when $\gamma$ is non-constant but fixes the class $A$, the second term in \eqref{MonodromySign} gives us the monodromy action of $\gamma$ on $C(q, A)$. Note that this action is trivial if the Maslov index of $A$ is even or if $\gamma$ is an orientation-preserving loop in $L$.}

As described in \cite[Proposition 7.4, Section 7.3]{Zapolsky}, a choice of relative pin structure defines an isomorphism between $\mathscr{C}(q, A)$ and $\mathscr{C}(q, A')$ whenever $I_{\mu_L}(A) = I_{\mu_L}(A')$, and quotienting by these identifications we obtain a local system $\mathscr{C}_\mu^*$ which has rank $1$ in degrees $\dots, -N_L^{\pi}, 0, N_L^{\pi}, 2N_L^{\pi}, \dots$ (and zero in every other degree).  Moreover, the choice defines a canonical trivialisation of the even degree part of $\mathscr{C}_\mu^*$, which then forms a constant sheaf of rings on $L$.  In more traditional versions of Floer theory the fibre of this sheaf of rings is viewed as the Novikov ring $\Lambda = R[T^{\pm 1}]$, were the variable $T$ has degree $N_L^{\pi}$ if $N_L^{\pi}$ is even (for instance if $L$ is orientable) and degree $2N_L^{\pi}$ otherwise.

If $N_L^{\pi}$ is odd then the odd degree part of $\mathscr{C}_\mu^*$ looks like $\Lambda \otimes \mathscr{C}_\mu^{N_L^{\pi}}$ and this local system may be twisted (contrast this with the even degree part, which is not just untwisted but canonically trivialised by the relative pin structure).  When the $\pi_1(L, q)$-action on $\pi_2(X, L, q)$ is trivial the twisting can be computed from \cref{lem:explicit monodromy} to be exactly $\det (TL)$, but when it is non-trivial the twisting may in general depend on the choice of relative pin structure.

The quotient procedure is compatible with the differential \cite[Section 7.3]{Zapolsky} so one obtains a cohomology group $HF^*_{\mathrm{Zap}, \mu}(L, L;R)$, and since it also respects the bigrading on the complex we still get the spectral sequence. In addition, quotienting is compatible with the multiplication \cite[Section 7.3]{Zapolsky} and with the quantum module structure \cite[Section 7.4]{Zapolsky} after applying a similar procedure to $QC^*_\mathrm{Zap}(X; R)$ \cite[Section 7.2]{Zapolsky}.  Although the latter does not depend on a choice of spin structure or similar on $X$, it does depend on the background class $b$ of the relative pin structure on $L$.  We denote this quotiented Zapolsky version of quantum cohomology by $QH^*(X, b; R)$ since it coincides with the standard $b$-twisted version of quantum cohomology over the Novikov ring $R[S^{\pm 1}]$, where $S$ has degree $2N^\pi_X$.  We think of $S$ as $T^d$, where $d=2N^\pi_X/|T|$, so that $R[S^{\pm1}]$ sits inside $\Lambda$.

In the setting relevant to us, the upshot of all of this is the following.

\begin{prop}
\label{prop Quotient Properties}
Suppose $L$ is relatively pin, that we fix a choice of relative pin structure on $L$ with background class $b \in H^2(X; \Z/2)$, and that $\pi_1(L, q)$ acts trivially on $\pi_2(X, L, q)$.  Then there are quantum cohomology $QH^*(X, b; R)$ and Floer cohomology rings $HF^*_{\mathrm{Zap}, \mu}(L, L; R)$ such that:
\begin{enumerate}[(a)]
\item\label{itma} If $N^\pi_L$ is even then $HF^*_{\mathrm{Zap}, \mu}(L, L; R)$ is naturally a module over the Novikov ring $R[T^{\pm 1}]$, where $T$ has degree $N^\pi_L$, and there is a spectral sequence
\[
E_1 = H^*(L; R[T^{\pm 1}]) \implies HF^*_{\mathrm{Zap}, \mu}(L, L; R).
\]
\item\label{itmb} If $N^\pi_L$ is odd then $HF^*_{\mathrm{Zap}, \mu}(L, L; R)$ is naturally a module over the Novikov ring $R[T^{\pm 1}]$, where $T$ has degree $2N^\pi_L$, and there is a spectral sequence
\[
E_1 = H^*(L; R[T^{\pm 1}]) \oplus H^*(L; \mathscr{R})[N^\pi_L] \implies HF^*_{\mathrm{Zap}, \mu}(L, L; R),
\]
where $\mathscr{R}$ is the rank $1$ local system of $R[T^{\pm 1}]$-modules on $L$ whose monodromy around a loop $\gamma$ is $(-1)^{\langle w_1(TL), \gamma \rangle_L}$ and where $[N^\pi_L]$ denotes grading shift.
\item\label{itmc} $QH^*(X, b; R)$ is a module over $R[T^{\pm d}]$, where $d = 2N^\pi_X / |T|$, and is additively isomorphic to $H^*(X; R[T^{\pm d}])$.  The product is deformed by counts of holomorphic spheres with signs twisted by $(-1)^{\langle b, A \rangle_X}$, where $A$ is the class of the sphere.
\item\label{itmd} There are PSS maps $H^*(X; R) \to QH^*(X, b; R)$ and $H^{\mathrm{PSS}}(L; R) \to HF^*_{\mathrm{Zap}, \mu}(L, L; R)$, where the latter comes from mapping part of the $E_1$ page of the above spectral sequences to the $E_\infty$ page (as described in \cref{classical comparison}), and a map $\CO : QH^*(X, b; R) \to HF^*_{\mathrm{Zap}, \mu}(L, L; R)$ of unital $R[T^{\pm d}]$-algebras which fits into a commutative diagram analogous to \eqref{CODiagram}.
\end{enumerate}
\end{prop}

\begin{rmk}
\label{rmk Char 2}
If $R$ has characteristic $2$ then we need not assume that $L$ is relatively pin, and choices of relative pin structure and background class are irrelevant.  Moreover, the local system $\mathscr{R}$ is trivial.
\end{rmk}

%This quotient procedure respects the bigrading on the complex and hence preserves the spectral sequence.  It is also compatible with the multiplication, and with the quantum module structure after applying a similar quotient procedure to $QC^*_\mathrm{Zap}(X; R)$.  Although the latter does not depend on a choice of spin structure or similar on $X$, it does depend on the background class of the relative \textcolor{red}{spin???} structure on $L$.

%---------------------------------------------------------------------------------------------
\subsection{Twisted coefficients}
\label{TwistedCoeffs}

As in traditional Lagrangian Floer theory, one can twist Zapolsky's Floer complex by local systems $\mathscr{E}^1$ and $\mathscr{E}^2$ of $R$-modules on $L$.  Now the complex is
\begin{equation}\label{eq def of HFzap with local coeffs}
CF^{r, s}_\mathrm{Zap} ((L, \mathscr{E}^1), (L, \mathscr{E}^2)) = \bigoplus_{\substack{q \in \mathrm{Crit}(f) \\ |q| = s}} \mathscr{C}^r_q \otimes C(q) \otimes \Hom_R (\mathscr{E}^1_q, \mathscr{E}^2_q).
\end{equation}
If a certain obstruction class vanishes (see \cite[Theorem 2.2.12]{KonstantinovThesis}, adapted to Zapolsky's setup in \cite[Appendix A.4]{SmithThesis}) then the differential squares to zero and we can take cohomology.  This is always the case when the minimal Maslov number of $L$ is at least $3$, which it will be in our applications.  One can also extend the Floer product to
\begin{equation}\label{eq product with local systems}
HF^*_\mathrm{Zap} ((L, \mathscr{E}^1), (L, \mathscr{E}^2)) \otimes_R HF^*_\mathrm{Zap} ((L, \mathscr{E}^0), (L, \mathscr{E}^1)) \rightarrow HF^*_\mathrm{Zap} ((L, \mathscr{E}^0), (L, \mathscr{E}^2)).
\end{equation}
The quotient procedure which defines $HF^*_{\mathrm{Zap}, \mu}$ only concerns the $\mathscr{C}^r_q$ factor of each summand of \eqref{eq def of HFzap with local coeffs} and so descends to this setting to give the complex 
\[
CF^{r, s}_{\mathrm{Zap}, \mu} ((L, \mathscr{E}^1), (L, \mathscr{E}^2)) = \bigoplus_{\substack{q \in \mathrm{Crit}(f) \\ |q| = s}} \mathscr{C}^r_{\mu, q} \otimes C(q) \otimes \Hom_R (\mathscr{E}^1_q, \mathscr{E}^2_q).
\]
 We denote its homology by $HF^*_{\mathrm{Zap}, \mu} ((L, \mathscr{E}^1), (L, \mathscr{E}^2))$.

We will only be interested in the situation where $\mathscr{E}^1$ is trivial and $\mathscr{E}^2$ is the local system corresponding to a cover $L'$ of $L$, in which case we denote the twisted Floer cohomology by $HF^*_\mathrm{Zap, \mu}(L, L'; R)$.  Floer theory for pairs of local systems of this form is essentially equivalent to Damian's lifted Floer homology \cite{Damian} on the cover $L'$.  However we choose to phrase it in the above way in order to fit it into the wider context of Floer theory with local systems, closed--open string maps, and Zapolsky's orientation schemes.

After incorporating local systems of this type, the analogue of \cref{prop Quotient Properties} is as follows.

\begin{prop}
\label{prop Twisted Properties}
Setup as in \cref{prop Quotient Properties}, with the additional assumption that $N^\pi_L \geq 3$.  For any cover $L'$ of $L$ we have a lifted Floer cohomology ring $HF^*_{\mathrm{Zap}, \mu}(L, L'; R)$ which satisfies the following modified versions of properties (\ref{itma})--(\ref{itmd}) from \cref{prop Quotient Properties}:
\begin{enumerate}[(A)]
\item\label{itmA} If $N^\pi_L$ is even then there is a spectral sequence
\[
E_1 = H^*_c(L'; R[T^{\pm1}]) \implies HF^*_{\mathrm{Zap}, \mu}(L, L'; R),
\]
where ${}_c$ denotes compact support, and $T$ has degree $N^\pi_L$.
\item\label{itmB} If $N^\pi_L$ is odd then there is a spectral sequence
\[
E_1 = H^*_c(L'; R[T^{\pm 1}]) \oplus H^*_c(L'; \mathscr{R}')[N^\pi_L] \implies HF^*_{\mathrm{Zap}, \mu}(L, L'; R),
\]
where $\mathscr{R}'$ is the pullback of $\mathscr{R}$ (from \cref{prop Quotient Properties}) to $L'$, and $T$ has degree $2N^\pi_L$.
\item\label{itmC} Identical to (\ref{itmc}).
\item\label{itmNewD} Identical to (\ref{itmd}).
\item\label{itmD} $HF^*_{\mathrm{Zap}, \mu}(L, L'; R)$ is a right module over $HF^*_{\mathrm{Zap}, \mu}(L, L; R)$, and thus over $QH^*(X, b; R)$ via $\CO$.
\end{enumerate}
\end{prop}
\begin{proof}[Sketch proof]
In the untwisted case (i.e.~in \cref{prop Quotient Properties}) the $E_1$ page is the (Morse) cohomology of $L$ with coefficients in the constant sheaf $R[T^{\pm1}]$, plus---when $N^\pi_L$ is odd---another copy of the cohomology of $L$ but with coefficients in $\mathscr{R}$.  Twisting by $\mathscr{E}^1$ (trivial) and $\mathscr{E}^2$ (corresponding to $L'$) amounts to tensoring the local systems with $\sheafHom (\mathscr{E}^1, \mathscr{E}^2)$, where $\sheafHom$ is sheaf hom (the $\sheafHom$ and tensor product are both taken over the constant sheaf $R$), so the $E_1$ page is given by
\[
H^*(L; \sheafHom (\mathscr{E}^1, \mathscr{E}^2) \otimes_R R[T^{\pm1}]) \; \text{or} \; H^*(L; \sheafHom (\mathscr{E}^1, \mathscr{E}^2) \otimes_R R[T^{\pm1}]) \oplus H^*(L; \sheafHom (\mathscr{E}^1, \mathscr{E}^2) \otimes_R \mathscr{R})[N_L^\pi]
\]
depending on the parity of $N^\pi_L$.  Strictly we are working with \emph{Morse} cohomology with local systems here, but it is proved in \cite{BanyagaHurtubiseSpaeth} that this coincides with other standard constructions of cohomology with local systems.

To identify this twisted cohomology, note that the singular chain complex of $L$ with coefficients in $\sheafHom (\mathscr{E}^1, \mathscr{E}^2) \otimes_R R[T^{\pm 1}]$ is equivalent to the ordinary singular chain complex of $L'$ over $R[T^{\pm 1}]$: both are freely generated over $R[T^{\pm 1}]$ by singular simplices in 
$L$ with lifts to $L'$.  By Poincar\'e duality we thus have
\[
H^*(L; \sheafHom (\mathscr{E}^1, \mathscr{E}^2) \otimes_R R[T^{\pm1}]) \cong H^*_c(L'; R[T^{\pm 1}]).
\]
Twisting the left-hand side by $\mathscr{R}$ corresponds to twisting the right by $\mathscr{R}'$.  This proves (\ref{itmA}) and (\ref{itmB}).

To get (\ref{itmD}) simply take $\mathscr{E}^0 = \mathscr{E}^1$ to be trivial in \eqref{eq product with local systems} to see that $HF^*_\mathrm{Zap}(L, L'; R)$ is a right (unital) module over $HF^*_{\mathrm{Zap}}(L,L;R)$ and hence also (via $\CO$) over quantum cohomology. Everything descends to the quotient determined by the choice of relative pin structure.
\end{proof}

%---------------------------------------------------------------------------------------------
\subsection{The example of $\R\P^n$}
\label{RPnExample}

To make all this more concrete we now describe how everything works for $(X, L) = (\C\P^n, \R\P^n)$ over the ground ring $R=\Z$, assuming $n \geq 2$. The Floer cohomology in this case was calculcated by Zapolsky in \cite[Section 8.1]{Zapolsky} but by different means. Our approach relies heavily on the spectral sequence and can serve as a guide to the proof of \cref{Theorem1} in \cref{section proof}, where our aim will be to show that any monotone Lagrangian $L$ in $\C\P^n$ of minimal Maslov number $n + 1$ behaves like $\R\P^n$.

%\textcolor{red}{Get rid of overlap with earlier.  Make hyperplane class name $h$ vs $H$ consistent}
%First note that $\C\P^n$ has minimal Chern number $2(n+1)$, and is simply connected with $\pi_2(\C\P^n, q)$ isomorphic to $\Z$ for any base point $q$.  This means that the local system $\Z[\pi_2(\C\P^n)]$ on which Zapolsky's quantum cohomology pearl complex lives is simply the constant sheaf with fibre the Novikov ring $\Z[T^{\pm 2}]$, where $T$ has degree $n+1$ (it might seem more natural to work with the Novikov ring $\Z[T^{\pm 1}]$, with $T$ assigned degree $2(n+1)$, but then we would have to introduce a square root of $T$ when we came to discuss the Floer cohomology of $\R\P^n$).  The quantum cohomology is then
%\[
%QH^*_\mathrm{Zap} (\C\P^n; \Z) \cong \Z[H, T^{\pm 2}]/(H^{n+1}-T^2),
%\]
%where $H$ is (the PSS image of) the hyperplane class.  Note that $H$ is invertible and of degree $2$, so by the quantum module action $HF^*_\mathrm{Zap}(\R\P^n, \R\P^n; \Z)$ must be $2$-periodic in its $\Z$-grading.

First observe that the map $H^2(\C\P^n; \Z/2) \rightarrow H^2(\R\P^n; \Z/2)$ is surjective (in fact, an isomorphism), since the restriction of the tautological complex line bundle on $\C\P^n$, whose second Stiefel--Whitney class generates $H^2(\C\P^n; \Z/2)$, is the direct sum of two copies of the tautological real line bundle on $\R\P^n$, whose squared first Stiefel--Whitney class generates $H^2(\R\P^n; \Z/2)$.  This means that $\R\P^n$ is relatively pin and thus satisfies assumption (O).

The group $\pi_2(\C\P^n, \R\P^n, q)$ is isomorphic to $\Z$ (for any base point $q$) and is generated by the class of a disc of Maslov index $n+1$. In particular, there is only one homotopy class of discs of each possible Maslov index, so we have that $\mathscr{C}^*_\mu$ coincides with $\mathscr{C}^*$ and the only purpose of the quotienting procedure from \cref{RelPin} is to identify the even degree part of $\mathscr{C}_\mu^*$ with the constant sheaf $R[T^{\pm 1}]$.  The action of $\pi_1(\R\P^n, q)$ on $\pi_2(\C\P^n, \R\P^n, q)$ is trivial, because it preserves Maslov index, so by \cref{lem:explicit monodromy} the odd degree part of $\mathscr{C}^*_\mu$ (which is only non-zero if $n$ is even) has monodromy $-1$ around the generator of $\pi_1(\R\P^n, q)$.

From this discussion, or directly from \cref{prop Quotient Properties}(\ref{itma}), if $n$ is odd then the first page of the Oh spectral sequence is
\[
\dots \rightarrow H^*(\R\P^n; \Z)[n] \rightarrow H^*(\R\P^n; \Z) \rightarrow H^*(\R\P^n; \Z)[-n] \rightarrow H^*(\R\P^n; \Z)[-2n] \rightarrow \dots
\]
(each term represents a column, the square brackets denote grading shift as usual, and the arrows represent the differential which maps horizontally from one column to the next), whilst if $n$ is even then by \cref{prop Quotient Properties}(\ref{itmb}) it is
\[
\dots \rightarrow H^*(\R\P^n; \mathscr{L})[n] \rightarrow H^*(\R\P^n; \Z) \rightarrow H^*(\R\P^n; \mathscr{L})[-n] \rightarrow H^*(\R\P^n; \Z)[-2n] \rightarrow \dots,
\]
where $\mathscr{L}$ denotes the unique non-trivial rank $1$ local system (over $\Z$) on $\R\P^n$.  Note that the cellular cochain complex computing $H^*(\R\P^n; \mathscr{L})$ is the same as that computing $H^*(\R\P^n; \Z)$ but the differentials which were $\pm 2$ become $0$ and vice versa.

The only potentially non-zero differentials in the whole spectral sequence are on this first page and go from $H^n$ in one column to $H^0$ in the adjacent column on the right.  Just focusing on these entries, the page is as shown in \cref{figOhSS}.
\begin{figure}[ht]
\begin{tikzpicture}
\begin{scope}[xshift=-3cm]
\matrix(m)[matrix of math nodes, nodes={inner sep=0, outer sep=0, minimum height=3.2ex, text depth=0ex, minimum width=0ex}, column sep=4.5ex]
{
& & & \ssvdots \\
& & \Z & \Z \\
& & \ssvdots & \\
& \Z & \Z & \\
& \ssvdots & & \\
\Z & \Z & & \\
\ssvdots & & & \\
};
\begin{scope}[every path/.style={->, shorten <= 0.15cm, shorten >= 0.15cm}]
\draw (m-2-3.east)--(m-2-4.west);
\draw (m-4-2.east)--(m-4-3.west);
\draw (m-6-1.east)--(m-6-2.west);
\end{scope}
\end{scope}

\begin{scope}[xshift=3cm]
\matrix(m)[matrix of math nodes, nodes={inner sep=0, outer sep=0, minimum height=3.2ex, text depth=0ex, minimum width=0ex}, column sep=4.5ex]
{
& & & \ssvdots \\
& & \Z & \Z \\
& & \ssvdots & \\
& \Z/2 & 0 & \\
& \ssvdots & & \\
\Z & \Z & & \\
\ssvdots & & & \\
};
\begin{scope}[every path/.style={->, shorten <= 0.15cm, shorten >= 0.15cm}]
\draw (m-2-3.east)--(m-2-4.west);
\draw (m-4-2.east)--(m-4-3.west);
\draw (m-6-1.east)--(m-6-2.west);
\end{scope}
\end{scope}
\end{tikzpicture}
\caption{The interesting part of the Oh spectral sequence for $\R\P^n \subset \C\P^n$ with $n$ odd (left) or even (right).\label{figOhSS}}
\end{figure}
The limit $HF^*_{\mathrm{Zap}, \mu}(\R\P^n, \R\P^n; \Z)$ must be $2$-periodic, by the quantum module action of $h \in QH^*(\C\P^n, b; \Z)$ via \cref{prop Quotient Properties}(\ref{itmd}), and $HF^2_{\mathrm{Zap}, \mu}(\R\P^n, \R\P^n; \Z) \cong H^2(\R\P^n;\Z)$ is $\Z/2$, so all of the $\Z \rightarrow \Z$ differentials have to be multiplication by $\pm 2$, and we obtain
\[
HF^p_{\mathrm{Zap}, \mu}(\R\P^n, \R\P^n; \Z) \cong \begin{cases} \Z/2 & \text{if } p \text{ is even} \\ 0 & \text{if } p \text{ is odd,}\end{cases}
\]
as computed by Zapolsky.  Note that, as expected, this is a module over the even degree part of $\Z[T^{\pm 1}]$, but it is \emph{not} a module over the whole ring if $n$ is even.  Note also that the Floer cohomology is $2$-torsion, as predicted by \cref{lemma: torsion of Floer cohomology}. 
%This is conveniently explained by quantum module action as follows.  The class $2H$ restricts to $0$ in $H^2(\R\P^n; \Z)$ (either by observing that $\R\P^n$ lies in the complement of a quadric representing the Poincar\'e dual of $2H$, or simply because $H^2(\R\P^n; \Z)$ itself is $2$-torsion), so $\CO(2H) = 0$ by \eqref{CODiagram}.  Since $\CO$ is a ring homomorphism and $H$ is invertible, this means that $\CO(2 \cdot 1_{\C\P^n}) = 2 \cdot 1_{\R\P^n}$ must vanish.  Here $1_X$ and $1_L$ denote the units (i.e.~multiplicative identity elements) in $QH^*_\mathrm{Zap}(X)$ and $HF^*_\mathrm{Zap}(L, L)$ respectively.

If we twist by the local system corresponding to the cover $S^n \rightarrow \R\P^n$ then by \cref{prop Twisted Properties}(\ref{itmA}) and (\ref{itmB}) the first page of the spectral sequence becomes
\[
\dots \rightarrow H^*(S^n; \Z)[n] \rightarrow H^*(S^n; \Z) \rightarrow H^*(S^n; \Z)[-n] \rightarrow H^*(S^n; \Z)[-2n] \rightarrow \cdots.
\]
The differentials must all be isomorphisms
%.  This can be computed either directly, by thinking about how the local systems interact with the differentials for $\R\P^n$, or indirectly as follows: 
since the resulting cohomology is still $2$-periodic (it is still a module over quantum cohomology by \cref{prop Twisted Properties}(\ref{itmD})) but the first page is zero in degrees $n-1$ and $n + 2$.%; if $n \geq 3$ then this immediately gives the result, whilst if $n=2$ one just has to think a little about what the differentials could be.

%---------------------------------------------

\subsection{Damian's argument}
\label{subsection Damian}

We conclude this review by outlining the proofs of \eqref{eqDamianOdd} and \eqref{eqDamianEven}, namely that if $L \subset \C\P^n$ is a monotone Lagrangian with $N_L = n+1$ then: for odd $n$, the universal cover $\widetilde{L}$ is homeomorphic to $S^n$ and $\pi_1(L)$ is finite; for even $n$, $\widetilde{L}$ is a $\Z/2$-homology sphere and $\pi_1(L) \cong \Z/2$.  This uses only standard (pre-Zapolsky) Floer theory---in the above language we work only with $\mathscr{C}_\mu^*$, which is trivialised either by fixing an orientation and spin structure on our Lagrangian or by using $\Z/2$ coefficients---and is intended to demonstrate the state of the art before the present work.  It should be contrasted with our new approach, detailed in \cref{section proof}, which completely circumvents the auxiliary construction of the circle bundle.

Let $\Gamma_L$ be the Biran circle bundle associated to $L$: view $\C\P^n$ as a hyperplane in $\C\P^{n+1}$ and take $\Gamma_L$ to be the lift of $L$ to the unit normal bundle of this hyperplane, embedded in $\C^{n+1} \cong \C\P^{n+1} \setminus \C\P^n$.  Consider its mod-$2$ Floer cohomology, lifted to its universal cover $\widetilde{\Gamma}_L$.  Since $\Gamma_L$ is displaceable this Floer cohomology must vanish.  On the other hand, it can be computed by a spectral sequence whose first page is $H^*_c(\widetilde{\Gamma}_L; \Z/2) \cong H_{n+1-*}(\widetilde{\Gamma}_L; \Z/2)$ (compare \cref{prop Twisted Properties}(\ref{itmA}) and (\ref{itmB}) and \cref{rmk Char 2}).  Combining these, we must have
\[
H^*_c(\widetilde{\Gamma}_L; \Z/2) \cong \begin{cases} \Z/2 & \text{if } * = 1 \text{ or } n+1 \\ 0 & \text{otherwise.} \end{cases}
\]
As observed by Schatz \cite[Lemmes 5.3--5.4]{SchatzThesis}, $\widetilde{\Gamma}_L$ is homotopy equivalent to $\widetilde{L}$, so we deduce that $\widetilde{L}$ is a $\Z/2$-homology sphere.  In particular, it is compact so $\pi_1(L)$ is finite.  Moreover, we have
\[
\chi_{\Z/2}(\widetilde{L}) = |\pi_1(L)| \cdot \chi_{\Z/2}(L),
\]
so if $n$ is even then $|\pi_1(L)| = 1$ or $2$.  It can't be $1$, because if $L$ were simply connected then it would have minimal Maslov number $2(n+1)$, so we must have $\pi_1(L) \cong \Z/2$.

When $n$ is odd, the fact that the minimal Maslov number of $\Gamma_L$ is even means that it is orientable (see \cref{remarkMuvsW1}), so its Floer theory can be defined over $\Z$ if it's also spin.  To prove that it is indeed spin we can apply the previous spectral sequence argument to the unlifted Floer cohomology of $\Gamma_L$ over $\Z/2$, and see that
\[
H^*(\Gamma_L; \Z/2) = H^*_c(\Gamma_L; \Z/2) \cong \begin{cases} \Z/2 & \text{if } * = 0 \text{, } 1 \text{, } n \text{, or } n+1 \\ 0 & \text{otherwise,} \end{cases}
\]
so the second Stiefel--Whitney class in $H^2(\Gamma_L; \Z/2)$ must vanish.  We can now go back to studying the Floer theory lifted to $\widetilde{\Gamma}_L$, but this time over $\Z$, and deduce that $\widetilde{L} \cong \widetilde{\Gamma}_L$ is a $\Z$-homology sphere.  Since $\widetilde{L}$ is also simply connected, the homology Whitehead theorem (see e.g. \cite{MayWhitehead}) and the Poincar\'e conjecture imply that $\widetilde{L}$ is homeomorphic to $S^n$.

%===========================================================
\section{Proof of \cref{Theorem1}}
\label{section proof}

\subsection{Preliminaries}

We now detail the proof of \cref{Theorem1}, which uses the full force of Zapolsky's machinery.  Throughout this section we fix the following setup: $L \subset \C\P^n$ is a closed, connected, monotone Lagrangian, of minimal Maslov number $n+1$.  The $n=1$ case of \cref{Theorem1} is trivial so we assume that $n$ is at least $2$.

Before embarking on any Floer theory, we make the following basic topological observation:

\begin{lem}
\label{H1mod2nonzero}
We have $H^1(L; \Z/2) \neq 0$.
\end{lem}
\begin{proof}
The long exact sequence in homology for the pair $(\C\P^n, L)$ gives the exact sequence
\[
H_2(\C\P^n; \Z) \rightarrow H_2(\C\P^n, L; \Z) \rightarrow H_1(L; \Z) \rightarrow 0.
\]
Applying the left-exact functor $\Hom_\Z(-, \Z/2)$ we obtain the exact sequence
\[
0 \rightarrow H^1(L; \Z/2) \xrightarrow{f} \Hom_\Z(H_2(\C\P^n, L; \Z), \Z/2) \xrightarrow{g} \Hom_\Z(H_2(\C\P^n; \Z), \Z/2),
\]
and the penultimate term contains the mod $2$ reduction $I'_{\mu_L}$ of $I_{\mu_L}/(n+1)$.  Since $I_{\mu_L}/(n+1)$ restricts to $2I_{c_1}/(n+1)$ on $H_2(\C\P^n; \Z)$, and this is always even, we deduce that $g(I'_{\mu_L})$ is zero.  This means that $I'_{\mu_L}$ is in the image of $f$, and as $I'_{\mu_L}$ itself is non-zero we must have $H^1(L; \Z/2) \neq 0$.
\end{proof}

Now consider the Floer cohomology $HF^*_{\mathrm{Zap}, \mu}(L, L; \Z/2)$. By \cref{prop Quotient Properties}(\ref{itma}) and (\ref{itmb}), and \cref{rmk Char 2}, it is computed by a spectral sequence whose $E_1$ page has $p$th column $H^*(L; \Z/2)[-pn]$.  For degree reasons the only potentially non-zero differentials in the whole spectral sequence map from $H^n(L; \Z/2)[-(p-1)n] \cong \Z/2$ to $H^0(L; \Z/2)[-pn] \cong \Z/2$ on this page, and these maps are independent of $p$ (the $p$-dependence is all contained in the local system $\mathscr{C}^*$, which we are quotienting down to $\mathscr{C}^*_\mu$, and in the signs which are irrelevant mod $2$) so it suffices to understand the $p=0$ case.

From this spectral sequence and the action of $\CO(h)$ (by \cref{prop Quotient Properties}(\ref{itmd})) we obtain the following two lemmas, which are analogous to \cite[Lemmas 6.1.3, 6.1.4]{BiranCorneaRigidityUniruling}.

\begin{lem}
\label{HFmod2nonzero}
We have an isomorphism of graded $\Z/2$-vector spaces
\[
HF^*_{\mathrm{Zap}, \mu}(L, L; \Z/2) \cong \bigoplus_{p=-\infty}^\infty H^*(L; \Z/2)[-p(n + 1)].
\]
That is $HF^k_{\mathrm{Zap}, \mu}(L, L;\Z/2) \cong \oplus_{p = -\infty}^\infty H^{k + (n + 1)p}(L;\Z/2) \cong H^{\ell_k}(L;\Z/2)$, where $\ell_k$ is the unique element of $\{0, \dots, n\}$ congruent to $k$ modulo $n+1$.
Further, $H^k(L;\Z/2) \cong \Z/2$ for all $0 \le k \le n$.
\end{lem}
\begin{proof}
By the preceding discussion, to prove the first part it is enough to show that the differential
\[
\partial_1 : H^n(L; \Z/2)[n] \rightarrow H^0(L; \Z/2)
\]
vanishes. Since the codomain comprises just $0$ and the classical unit, we are done if the latter survives the spectral sequence.  But (the PSS image of) the classical unit is also the unit $1_L$ for the Floer product so we simply need to check that $HF^*_{\mathrm{Zap}, \mu}(L, L; \Z/2)$ is non-zero.  To see that this is indeed the case, observe that $H^1(L; \Z/2)$ survives and is non-zero by \cref{H1mod2nonzero}. 

We thus have that $HF^*_{\mathrm{Zap}, \mu}(L, L; \Z/2) \cong \bigoplus_{p=-\infty}^\infty H^*(L; \Z/2)[-p(n + 1)]$. 
In particular, we see that $HF^0_{\mathrm{Zap}, \mu}(L, L;\Z/2) \cong H^0(L;\Z/2) \cong \Z/2$ and $HF^{-1}_{\mathrm{Zap}, \mu}(L, L;\Z/2) \cong H^n(L;\Z/2) \cong \Z/2$. But by invertibility of the hyperplane class $h$ in quantum cohomology, Floer multiplication by $\CO(h)$ gives an isomorphism $HF^k_{\mathrm{Zap}, \mu}(L, L;\Z/2) \cong HF^{k + 2}_{\mathrm{Zap}, \mu}(L, L;\Z/2)$ for every $k \in \Z$ and so we must have $HF^k_{\mathrm{Zap}, \mu}(L, L;\Z/2) \cong \Z/2$ for all $k \in \Z$. This finishes the proof.
\end{proof} 

\begin{lem}
\label{H2}
The group $H^2(L; \Z/2)$ is isomorphic to $\Z/2$ and is generated by $i^*h$, where $i : L \rightarrow \C\P^n$ is the inclusion.  In particular, $L$ is relatively pin and hence satisfies assumption (O).
\end{lem}
\begin{proof}
By \cref{HFmod2nonzero} we already know that $H^2(L; \Z/2)$ is isomorphic to $\Z/2$.  Moreover, its proof simultaneously shows the following:
\begin{itemize}
\item the map $\mathrm{PSS}\colon H^2(L;\Z/2) \to HF^2(L, L;\Z/2)$ is an isomorphism (it also shows that this map is well-defined when $n = 2$, i.e. that $H^2(L;\Z/2) \le H^\mathrm{PSS}(L;\Z/2)$);
\item Floer multiplication by $\CO(h)$ gives an isomorphism
\[
HF^0_{\mathrm{Zap}, \mu}(L, L; \Z/2) \xrightarrow{\ \sim \ } HF^2_{\mathrm{Zap}, \mu}(L, L; \Z/2);
\]
\item $HF^0_{\mathrm{Zap}, \mu}(L, L; \Z/2)$ is $\Z/2$, generated by the unit $1_L$.
\end{itemize}
From the latter two items we get that $HF^2_{\mathrm{Zap}, \mu}(L, L;\Z/2)$ is $\Z/2$, generated by $\CO(h) * 1_L = \CO(h)$.
The diagram in \cref{prop Quotient Properties}(\ref{itmd}), relating $i^*$ to $\CO$, then yields the commuting diagram
\begin{equation}
\begin{tikzcd}
H^2(\C\P^n; \Z/2) \arrow{r}{i^*} \arrow{d}[swap]{\mathrm{PSS}} & H^2(L; \Z/2) \arrow{d}{\mathrm{PSS}}
\\ QH^2(\C\P^n; \Z/2) \arrow{r}{\CO} & HF^2_{\mathrm{Zap}, \mu}(L, L; \Z/2)
\end{tikzcd}
\end{equation}
The left-hand vertical map is an isomorphism between $\Z/2$'s, and the above discussion shows that the same is true for the right-hand vertical map and the bottom horizontal map. Hence the top horizontal map is also an isomorphism, which is what we wanted.
\end{proof}

Observe that \cref{HFmod2nonzero} allows us to immediately complete the $n = 2$ case of \cref{Theorem1} since, by the classification of surfaces, $\R\P^2$ is the only closed surface whose first cohomology group with $\Z/2$ coefficients is isomorphic to $\Z/2$.

We are now ready to unleash Floer theory over $\Z$ in order to deal with the general case.
\subsection{The main argument}

From now on we assume that $n \geq 3$, and fix an arbitrary choice of relative pin structure on $L$.

Since assumption (O) is satisfied we can work over $\Z$, and since $N_L \geq 3$ we can twist by any cover $L'$ as in \cref{TwistedCoeffs}, and consider the cohomology $HF^*_{\mathrm{Zap}, \mu}(L, L';\Z)$.  In principle this depends on the choice of relative pin structure but we do not explicitly notate this. By \cref{prop Twisted Properties}(\ref{itmA}) and (\ref{itmB})  the zeroth column of the first page of the Oh spectral sequence which computes $HF^*_{\mathrm{Zap}, \mu}(L, L';\Z)$ is isomorphic to $H^*_c(L'; \Z)$, and for degree reasons all of the intermediate cohomology (meaning $0 < * < n$) survives.  The key result is the following:

\begin{prop}
\label{MainProp}
For any cover $L'$ of $L$ the compactly-supported cohomology groups $H^k_c(L';\Z)$ for $0 < k < n$ are $2$-torsion and $2$-periodic. 
%In particular $H^{2k + 1}(L';\Z) = 0$ for all $0 \le k < \lfloor\frac{n}{2}\rfloor$.
\end{prop}
\begin{proof}
Since these intermediate cohomology groups survive to $HF^*_{\mathrm{Zap}, \mu}(L, L'; \Z)$, they are acted upon by the invertible element $\CO(h)$ of degree $2$. This gives us $2$-periodicity.

To prove $2$-torsion, first note that since the minimal Maslov number is greater than $2$, the well-known argument of Auroux, Kontsevich and Seidel (\cite[Proposition 6.8]{AurouxMSandTduality}, \cite[Lemma 2.7]{SheridanFano}) implies that $\CO(2c_1(\C\P^n)) = 0$, that is $2(n+1)\CO(h) = 0$.  
Another way to see this is to note that $\CO(2c_1(\C\P^n)) = \mathrm{PSS}(i^*(2c_1(\C\P^n)))$ which vanishes since $2c_1(\C\P^n) = j^*(\mu_L)$ by \eqref{eq Viterbo}.

Taking $L' = L$ and applying invertibility of $h$ again, we see that $HF^*_{\mathrm{Zap}, \mu}(L, L;\Z)$, and hence the intermediate cohomology of $L$, is $2(n+1)$-torsion. But by the universal coefficients theorem, $H^1(L; \Z)$ is torsion-free, so it must vanish. 
%Combining this with $2$-periodicity, we get $H^{2k + 1}(L';\Z) = 0$ for all $0 \le k < \lfloor\frac{n}{2}\rfloor$.
We can now apply \cref{lemma: torsion of Floer cohomology} to see that $HF^*_{\mathrm{Zap}, \mu}(L, L; \Z)$ is $2$-torsion. By \cref{prop Twisted Properties}(\ref{itmD}), for each cover $L'$ the cohomology $HF^*_{\mathrm{Zap}, \mu}(L, L'; \Z)$ is a (unital) module over the ring $HF^*_{\mathrm{Zap}, \mu}(L, L; \Z)$, so the former must also be $2$-torsion.  This in turn means that the intermediate compactly-supported cohomology groups of each $L'$ are $2$-torsion.
\end{proof}

Finally we complete the proof of \cref{Theorem1}, by showing that $L$ has fundamental group $\Z/2$ and universal cover homeomorphic to $S^n$:

\begin{proof}[Proof of \cref{Theorem1}]
%Fix an arbitrary basepoint $q$ in $L$, and 
Let $L^+$ denote the minimal orientable cover of $L$, meaning $L$ itself, if it is orientable, or the orientable double cover otherwise. 
Apply \cref{MainProp} to every connected cover $L'$ of $L^+$ to see that for every such cover the group $H^{n - 1}_c(L';\Z)$ is $2$-torsion. Since $L'$ is orientable, Poincar\'{e} duality tells us that $H^{n - 1}_c(L';\Z)$ is isomorphic to $H_1(L';\Z)$ and so the latter is $2$-torsion.

 By the Hurewicz theorem, this means that every subgroup of $\pi_1(L^+)$ has $2$-torsion abelianisation.  In particular, by considering the cyclic subgroups, we see that every element of $\pi_1(L^+)$ has order $2$, so the group is abelian (every commutator $aba^{-1}b^{-1}$ is square $(ab)^2$ and hence equal to the identity). We deduce that $\pi_1(L^+)$ is isomorphic to $H_1(L^+; \Z)$ and is $2$-torsion. It is also finitely-generated (since $L^+$ is compact) and therefore finite. Hence $\pi_1(L)$ is finite as well.
 
Consider now the universal cover $\widetilde{L}$ of $L$, which is compact by the above discussion. By Hurewicz and universal coefficients, $H^1(\widetilde{L};\Z)$ vanishes and $H^2(\widetilde{L};\Z)$ is torsion-free. Then \cref{MainProp} shows that $\widetilde{L}$ is an integral homology sphere, so as in \cref{subsection Damian} it is homeomorphic to $S^n$.

We are now in a position to finish the proof by showing that $\pi_1(L)$ is $\Z/2$. Suppose first that $L$ is orientable, in which case we replace $L^+$ by $L$ in the above to see that $\pi_1(L)$ is finite, abelian and $2$-torsion and hence $\pi_1(L) \cong (\Z/2)^k$ for some $k \in \mathbb{N}$. On the other hand, by \cref{HFmod2nonzero}, we know that $H^1(L;\Z/2) \cong \Z/2$ and so $k = 1$.
Suppose now that $L$ is non-orientable. Then by \cref{remarkMuvsW1} we see that $n$ must be even. But then $\pi_1(L)$ is a non-trivial group which acts freely on the even-dimensional sphere $\widetilde{L}$ and so we must have $\pi_1(L) \cong \Z/2$. 
\end{proof}

\bibliography{RPnBibliography}
\bibliographystyle{utcapsor2}

\end{document}